\documentclass{amsart}
\usepackage{amsmath, amscd, amssymb, amsthm}
\usepackage{bbm}
\usepackage{latexsym}
\usepackage{amsfonts}
\usepackage{graphicx}
\usepackage[all,cmtip]{xy}

\usepackage[
dvipdfm, colorlinks=true, linkcolor=black,breaklinks=true,
urlcolor=blue, citecolor=black
]{hyperref}%
\usepackage{geometry}
\newtheorem{theorem}{Theorem}%[section]
\newtheorem{lemma}{Lemma}

\newtheorem{proposition}[theorem]{Proposition}

\newtheorem{question}[theorem]{Question}

\newtheorem*{definition}{Definition}

\newcommand{\be}{\begin{equation}}
\newcommand{\ee}{\end{equation}}
\newcommand{\bea}{\begin{eqnarray}}
\newcommand{\eea}{\end{eqnarray}}

\newcommand{\vs}{\vspace{0.5cm}}

\newcommand{\vsv}{\vspace{0.12cm}}

\def\XXint#1#2#3{{\setbox0=\hbox{$#1{#2#3}{\int}$ }
\vcenter{\hbox{$#2#3$ }}\kern-.6\wd0}}

\begin{document}

\title{The set of all orthogonal complex\\ structures on the flat $6$-tori}
\author{Gabriel Khan}

\address{Gabriel Khan. Department of Mathematics,
The Ohio State University, 231 West 18th Avenue, Columbus, OH 43210,
USA}
\email{{khan.375@osu.edu}}

\author{Bo Yang}
%\thanks{Research partially supported by an AMS-Simons Travel Grant}

\address{Bo Yang. Department of Mathematics, Rutgers University,
110 Frelinghuysen Road, Piscataway, NJ 08854, USA.}
\email{{boyang@math.rutgers.edu}}

%\author{Fangyang Zheng} \thanks{Research partially supported by NSFC 11271320}
%\address{Fangyang Zheng. Department of Mathematics, Zhejiang Normal University,
%Jinhua, 321004, Zhejiang, China and Department of Mathematics,

\author{Fangyang Zheng}
%\thanks{The third-named author is partially supported by a Simons
%Collaboration Grant}

\address{Fangyang Zheng. Department of Mathematics, The Ohio State
University, 231 West 18th Avenue, Columbus, OH 43210, USA}
\email{{zheng.31@osu.edu}}

\begin{abstract}
In \cite{BSV}, Borisov, Salamon and Viaclovsky constructed non-standard
orthogonal complex structures on flat tori $T^{2n}_{\mathbb R}$ for any
$n\geq 3$. We will call these examples BSV-tori. In this note, we show that
on a flat $6$-torus, all the orthogonal complex structures are either the
complex tori or the BSV-tori. This solves the classification problem for
compact Hermitian manifolds with flat Riemannian connection in the case
of complex dimension three.

\end{abstract}

\maketitle

\tableofcontents

\markleft{The set of all orthogonal complex structures on the flat
$6$-tori} \markright{The set of all orthogonal complex structures on
the flat $6$-tori}

\section{Introduction}

Given a Hermitian manifold $(M^n,g)$, there are several canonical metric
connections on it that are well-studied. The Riemannian (or Levi-Civita)
connection $\nabla$ which is torsion free, and the Chern (aka Hermitian)
connection $\nabla^c$ which is compatible with the complex structure, and
the Bismut connection $\nabla^b$, which is compatible with the almost
complex structure and has skew-symmetric $(3,0)$ torsion.
When $g$ is K\"ahler, all three connections coincide, but when $g$ is
not K\"ahler, the three are  mutually distinct. Let us denote by $R$,
$R^c$, and $R^b$ the corresponding curvature tensors, respectively.

\vsv

From the differential geometric point of view, it is very natural to
study the curvature of each of these connections, and ask what kind of
manifolds are ``space forms" with respect to a given connection. In
particular, one could ask what kind of compact complex manifolds will
 admit a Hermitian metric with flat Riemannian or Chern or Bismut
 connection?

\vsv

For the Chern connection $\nabla^c$, Boothby \cite{Boothby} proved in 1958
that compact Hermitian manifolds with $R^c=0$ identically are exactly the
compact quotients of complex Lie groups equipped with left invariant
metrics. Such manifolds can be non-K\"ahler when $n\geq 3$. H.-C. Wang's
complex parallisable manifolds \cite{Wang} form an important subset in
this class.

\vsv

For the Bismut connection $\nabla^b$, in a recent work \cite{WYZ}, we were
able to show that compact Hermitian manifolds $(M^n,g)$ with flat Bismut
connections are exactly those covered by Samelson spaces, namely,
$G\times {\mathbb R}^k$ equipped with a bi-invariant metric and a
left invariant complex structure. Here $G$ is a simply-connected compact
semisimple Lie group, and $0\leq k \leq 2n$. In particular, compact
non-K\"ahler Bismut flat surfaces are exactly those isosceles Hopf
surfaces, and in dimension three their universal cover is either a central
Calabi-Eckmann threefold $S^3\times S^3$, or
$({\mathbb C}^2\setminus \{ 0\} )\times {\mathbb C}$. We refer the readers
to \cite{WYZ} for more details.

\vsv

So now we are left with the question of answering what kind of
compact Hermitian manifolds $(M^n,g)$ will have identically zero
Riemannian curvature tensor?  By Bieberbach Theorem, we know that
such manifolds admit finite unbranched cover that is a flat torus
$T^{2n}_{\mathbb R}$. So the question boils down to what kind of
orthogonal complex structures are there on a flat $T^{2n}_{\mathbb
R}$?

\vsv

Given a flat $2n$-torus $M=T^{2n}_{\mathbb R}$, first of all, there
are always compatible complex structures $J$ on $M$ that makes $M$ a
complex $n$-torus. All such complex structures (compatible with the
orientation) are parameterized by the Hermitian symmetric space
$Z_n=SO(2n)/U(n)$. Clearly, for a complex structure $J$ on $M$
compatible with the flat metric $g$, if $J$ makes $g$ a K\"ahler
metric, then $(M,J)$ is a complex torus. In this case we will call
this $J$ a standard complex structure. When $J$ makes the metric $g$
non-K\"ahler, we will call such a complex structure non-standard.

\vsv

When $n=2$, the classification theory for compact complex surfaces
implies that any complex structure on $T^4_{\mathbb R}$ must be a
complex $2$-torus, thus there are no non-standard complex structures.
For $n\geq 3$, however, there are non-standard complex structures on
some flat $2n$-torus for each $n\geq 3$. In \cite{BSV}, Borisov, Salamon,
and Viaclovsky constructed non-standard orthogonal complex structures
on some flat $T^{2n}_{\mathbb R}$ for any $n\geq 3$. We will call these
examples {\em warped tori of Borisov-Salamon-Viaclovsky,} or {\em BSV-tori}
 for short. In Section \ref{BSV systematic}, we will give some explicit
 discussion of BSV-tori in dimension $3$ and their generalizations. In
 particular, BSV-tori in dimension $3$ are defined as follows:

\begin{definition} [\textbf{BSV $3$-tori}]
For $i=1$ and $2$, let $(M_i, g_i)$ be the flat torus of real
dimension $2$ and $4$, respectively, and let $(M,g)$ be their
product. Let $J_1$ be the complex structure determined by $g_1$,
which makes $M_1$ an elliptic curve. Let $f$  be a non-constant
holomorphic map $f: M_1 \rightarrow {\mathbb P}^1$. Since ${\mathbb
P}^1=SO(4)/U(2)$ is the set of all complex structures on the flat
$4$-torus $(M_2, g_2)$ compatible with the metric and the
orientation, one may consider almost complex structures $J $ on $M$
defined by
$$J= J_1+J_{f(y_1)}$$
at the point $(y_1, y_2)$ in $M=M_1\times M_2$. It is shown in \cite{BSV}
that $J$ is integrable since $f$ is holomorphic, so $(M,g,J)$ becomes
a Hermitian manifold with everywhere zero Riemannian curvature. The
metric $g$ is not K\"ahler with respect to these complex structures
(since $f$ is non-constant), so they are all non-standard.
\end{definition}

Note that any BSV-$3$-torus is always a product of a flat $2$-torus with
a flat $4$-torus as a Riemannian manifold, while a generic flat $6$-torus
 does not split. Also, as a complex manifold, a BSV $3$-torus $M^3$ is
 a holomorphic submersion over an elliptic curve, whose fibers are
 complex $2$-tori, but the fibers are not all biholomorphic to each other.

\vsv

The main purpose of this article is to show that, in complex dimension
three,  BSV-tori actually give all the possible orthogonal complex
structures on the flat torus $T^{6}_{\mathbb R}$, besides the standard
complex tori. In other words, we have the following:

\begin{theorem} \label{Riemannianflat}
Let $(M^3,g)$ be a compact Hermitian manifold whose Riemannian curvature
tensor is identically zero. Then a finite unbranched cover of $M$ is
holomorphically isometric to either a flat complex torus or a BSV-torus.
\end{theorem}

As the proof shall indicate, in higher dimensions, Riemannian flat
compact Hermitian manifolds are still rather special and should form
 a highly restrictive class which contains all BSV tori. But perhaps
  a generalization of BSV tori should be formulated and organized
  before a classification statement can be made and proved.
  For $n\geq 4$, the algebraic behavior of the Chern torsion tensor
   is much more complicated than the $n=3$ case, and we intend to
   pursue these higher dimensional cases as the next project.

\vsv

One property worth noticing is that, these BSV $3$-tori are actually
non-K\"ahlerian, namely, they do not admit any K\"ahler metric:

\begin{proposition} \label{torsionparallel}
Let $M^3=M_1\times M_2$ be a BSV $3$-tori, where $M_1$ is a flat
$2$-torus and $M_2$ a flat $4$-torus. Then $M$ admits no
pluri-closed Hermitian metrics, in particular, it is non-K\"ahlerian. Its Kodaira dimension is $-\infty $, and
its total torsion, namely, the $L^2$-norm of the Chern torsion of $M$ with respect to the
standard flat metric $g$,  is equal to $32\pi v_2d$, where $v_2$ is the volume
of $M_2$ and $d$ the degree of the map $f: M_1 \rightarrow {\mathbb
P}^1$.
\end{proposition}

In Section \ref{BSV systematic} we will prove a slightly more general version
of Proposition \ref{torsionparallel}, where $M_1$ is replaced by any
compact Riemann surface with positive genus. All statements are valid except the one on
Kodaira dimension. We should point it out it is already proved in
\cite{BSV} (Proposition 5.3 on P.144 \cite{BSV}) that the flat metric $(M^3, J, g)$
is not K\"ahler  if the holomorphic map $f$ in the definition is
non-constant. Here we emphasize that $(M^3, J)$ is non-K\"ahlerian in
the sense that it does not admit any K\"ahler metric.

Since the degree of the map $f$ can be any positive integer
greater than $1$, we know that on $T^6_{\mathbb R}$, there are
infinitely many complex structures with mutually distinct first
Chern class, and there is no uniform bound on the total torsion,
even though all complex structures are balanced in this case
(\cite{Gauduchon2}, \cite{BSV}).

\vsv

In 1958 Calabi \cite{Calabi} discovered that $M_1 \times
T_{\mathbb{R}}^4$ where $M_1$ is a hyperelliptic Riemann surface
with odd genus $g \geq 3$ and $T_{\mathbb{R}}^4$ a real $4$-torus,
can be given a complex structure $J$ such that the resulting
threefold $(M^3, J)$ admits no K\"{a}hler metric and has vanishing
fist Chern class. The complex structure Calabi used is related to
vector cross product in the space of purely Cayley numbers. In
Section \ref{Calabi 3-folds revisited}, we show that Calabi's
construction is a special case of the BSV type warped complex
structures on $M_1 \times T_{\mathbb{R}}^4$. The induced Hermitian
metrics from Calabi's construction is also a special case of
balanced metrics which are product Riemannian metrics.

\vsv

It seems natural to ask whether Theorem \ref{Riemannianflat} is also
true when $M_1$ is a Riemann surface with genus $g \geq 2$ with its
standard hyperbolic metric. In the end of paper we formulate the
problem and leave it to the future studies.

\vs

\section{The kernel spaces of the torsion}

\vsv

Let us start with a  Hermitian manifold $(M^n,g)$. Following the notations of \cite{YZ}, we will denote by $\nabla$, $\nabla^c$ the Riemannian (aka Levi-Civita) or the Chern (aka Hermitian) connection, respectively. Denote by $R$, $R^c$ the curvature tensors of these two connections, and by $T^c$ the torsion tensor of $\nabla^c$. Under a local unitary frame $e$ of type $(1,0)$ tangent vectors, $T^c$ has components
$$ T^c(e_i, e_j) = \sum_{k=1}^n 2\ T_{ij}^k e_k, \ \ \ \ \  \ T^c(e_i, \overline{e_j}) =0. \ \ \ \ \ $$

By Lemma 7 of \cite{YZ}, we have the following
\begin{eqnarray}
2T^k_{ij,\ \overline{l}} & = & R^c_{j\overline{l}i\overline{k}} -
R^c_{i\overline{l}j\overline{k}} \ ,
\label{formula 21}\\
R_{ijk\overline{l}} \ & = & T^l_{ij,k} + T^l_{ri} T^r_{jk} -
T^l_{rj} T^r_{ik} \ ,  \label{formula 22} \\
R_{ij\overline{k}\overline{l}} & = &
T^l_{ij,\overline{k}} - T^k_{ij,\overline{l}} +
2T^r_{ij} \overline{T^r_{kl}} + T^k_{ri} \overline{ T^j_{rl} } +
T^l_{rj} \overline{ T^i_{rk} } - T^l_{ri} \overline{ T^j_{rk} } -
T^k_{rj} \overline{ T^i_{rl} } \ ,   \label{formula 23}\\
R_{k\overline{l}i\overline{j}} & = &
R^c_{k\overline{l}i\overline{j}} - T^j_{ik,\overline{l}} -
\overline{ T^i_{jl,\overline{k}} } + T^r_{ik} \overline{ T^r_{jl} }
- T^j_{rk} \overline{ T^i_{rl} } - T^l_{ri} \overline{ T^k_{rj} } \ ,
\label{formula 24}
\end{eqnarray}
for any indices $i$, $j$, $k$, $l$. Here and below, $r$ is summed from $1$ to $n$,  and the index after the comma stands for covariant derivative with respect to $\nabla^c$.

\vsv

Now let us denote by $T^k_{ij;l}$, $T^k_{ij; \overline{l}}$ the covariant derivatives with respect to $\nabla$. Following the notations of \cite{YZ}, we have
\begin{eqnarray}
 \nabla_{e_l}e_i & = & \nabla^c_{e_l}e_i + \gamma_{ir}(e_l) \ e_r \ = \    \nabla^c_{e_l}e_i +  T^r_{il} \ e_r \ ,\\
 \nabla_{\overline{e_l}} e_i & = & \nabla^c_{\overline{e_l}} e_i + \gamma_{ir}(\overline{e_l}) \ e_r + \overline{(\theta_2)_{ir}(e_l)} \ \overline{e_r} \ = \  \nabla^c_{\overline{e_l}} e_i - \overline{T^i_{rl}} \ e_r + T^l_{ir} \ \overline{e_r} \ .
\end{eqnarray}

By a straight forward computation, we obtain the following identities:
\begin{eqnarray}
T^k_{ij;l } & = & T^k_{ij,l} - T^k_{rj}T^r_{il} - T^k_{ir} T^r_{jl} + T^r_{ij} T^k_{rl} \ , \\
T^k_{ij;\overline{l}} & = & T^k_{ij,\overline{l}} - T^r_{ij} \overline{T^r_{kl}} + T^k_{rj} \overline{T^i_{rl} } - T^k_{ri} \overline{T^j_{rl} } \ .
\end{eqnarray}

From the last equality, we get
\begin{equation}
T^k_{ij;\overline{l}} - T^l_{ij;\overline{k}} = T^k_{ij,\overline{l}} - T^l_{ij,\overline{k}} - 2 T^r_{ij} \overline{T^r_{kl}} +  T^k_{rj} \overline{T^i_{rl} } - T^k_{ri} \overline{T^j_{rl} } -  T^l_{rj} \overline{T^i_{rk} } + T^l_{ri} \overline{T^j_{rk} } \ .
\end{equation}

Combining (2) and (7), or comparing (3) with (9), we get
\begin{eqnarray}
 & & T^k_{ij;l } \ =  \ T^r_{ij} T^k_{rl} + R_{ijl\overline{k}} \ , \\
 & & T^k_{ij;\overline{l}} - T^l_{ij;\overline{k}} \ = \ - R_{ij\overline{k}\overline{l}} \ .
\end{eqnarray}

So for Hermitian manifold $(M^n,g)$ with $R=0$ everywhere,  we have the following
\begin{lemma}
On a Hermitian manifold $(M^n,g)$ with identically zero Riemannian curvature, let $T^k_{ij}$ be the components of (half of) the torsion of the Chern connection, under a local unitary frame $e$. Their covariant derivatives with respect to the Riemannian connection $\nabla$ satisfy
\begin{eqnarray}
 & & T^k_{ij;l } \ = \ \sum_r T^r_{ij} T^k_{rl}  \ ,\\
 & & T^k_{ij;\overline{l}} \ = \ T^l_{ij;\overline{k}}
 \end{eqnarray}
 for any $i$, $j$, $k$, $l$ between $1$ and $n$.
  \end{lemma}

Now if $M$ is also compact, then since $R=0$,  by the equality case of the main theorem of \cite{Gauduchon2}, or by Theorem 3 of \cite{YZ}, we know that $M$ is balanced. That is, $\sum_l T^l_{il} =0$ for any $i$.  So by (12) we have $\sum_l T^l_{ij;l}=0$ for any $i$, $j$.

 \vsv

Let us fix a point $p\in M^n$. Denote by $W\cong {\mathbb R}^{2n}$ the real tangent space of $M$ at $p$, and by $V\cong {\mathbb C}^n$ the space of type $(1,0)$ complex tangent vectors at $p$, and $J$ the almost complex structure of $M$. Since $T^c(\overline{e_i}, e_j)=0$ and $T^c(e_i,e_j)=2 \sum_{k=1}^n T^k_{ij}e_k$ under any unitary frame $e$, we have
\begin{equation*}
T^c(Jx,y)=T^c(x,Jy), \ \ \ \ T^c(Jx,y) = JT^c(x,y)
\end{equation*}
for any $x$, $y$ in $W$. Consider linear subspaces $K_1$, $K_2$ in $W$ defined by
\begin{eqnarray*}
 K_1 & = & \{ x\in W\ \mid \  T^c(x,u)=0 \ \ \forall \ u \in W \} , \\
 K_2 & = & \{ x\in W\ \mid \  \langle T^c(u,v), x \rangle =0 \ \forall \ u, v \in W \} .
 \end{eqnarray*}
Clearly $K_1$, $K_2$ are both $J$-invariant. Let $K_0=K_1\cap K_2$, and for $i=1$, $2$, write $K_i'=K_0^{\perp } \cap K_i$. Then we have orthogonal decomposition $K_i=K_0 \oplus K_i'$ for $i=1$, $2$. We claim that

\begin{lemma}
If the components of the torsion tensor under a unitary frame $e$ at $p$ satisfy the condition
\begin{equation}
\sum_{r=1}^n  T^r_{ij} T^k_{rl} =0
\end{equation}
for any $i$, $j$, $k$, $l$, then at the point $p$ we have the orthogonal decomposition $$W=K_0 \oplus K_1' \oplus K_2'.$$
\end{lemma}

\begin{proof}
Note that all the subspaces $K_0$, $K_i$, and $K_i'$ are $J$-invariant, so we may consider their corresponding complex subspaces $N_0$, $N_i$, and $N_i'$ in $V$ instead, where $i=1$, $2$. Clearly, $N_1$ consists of all $X\in V$ such that $T_{X\ast }^{\ast }=0$, and $N_2$ consists of all $X\in V$ such that $T^X_{\ast \ast }=0$.

\vsv

Here and from now on we adopted the convention that $T^X_{ij} = \sum_k \overline{X_k} T^k_{ij}$ for $X= \sum_k X_k e_k$ in $V$. This is because $T^k_{ij}$ is conjugate linear in the upper position.

\vsv

As in the proof of Theorem 2 of \cite{YZ}, for $X=\sum_i X_ie_i \in V$, we will denote by $A_X$ the linear transformation from $V$ to $V$ defined by
$$ A_X (e_i) = \sum_{j=1}^n T^j_{Xi} e_j = \sum_{k,j=1}^n X_k T^j_{ki} e_j\ .$$

With this notation, $(14)$ is simply saying that $A_XA_Y=0$ for any $X$, $Y$ in $V$. In particular, $(A_X)^2=0$.  So $N_2$ is the orthogonal complement of $\sum_{X\in V} \mbox{Im}(A_X)$, where $ \mbox{Im}(A_X)$ stands for the image space of $A_X$. In the mean time, it is clear that $N_1=\bigcap_{X\in V} \ker (A_X)= \{ X\in V \mid A_X=0\}$.

\vsv

Since $A_XA_Y=0$ for any $X$, $Y$ in $V$, we have $ \sum_{X\in V} \mbox{Im}(A_X) \subseteq \bigcap_{X\in V} \ker (A_X)$. So $N_2^{\perp } \subseteq N_1$. Therefore, $V=N_0\oplus N_1' \oplus N_2'$, where $N_1=N_0\oplus N_1'$, $N_2=N_0\oplus N_2'$, and all the direct sums are orthogonal. This completes the proof of the lemma.
\end{proof}

\noindent {\em Remark:}
(1). This lemma says that, when the equation $(14)$ holds, or equivalently $T^k_{ij; l}=0$ by $(12)$, the torsion tensor obeys a nice decomposition which resembles those on a warped torus of the BSV type \cite{BSV}.

\vsv

(2). Notice that for any $0\neq X\in N_1'$, there exists some $Y$, $Z$ in $V$ (necessarily in $N_2'$) such that $T^X_{YZ}\neq 0$, as otherwise $X$ would be in $N_2$ by definition. Similarly, for any $0\neq X\in N_2'$, there must be $Y$ and $Z$ (where $Y\in N_1'$ and $Z\in N_2'$ necessarily) such that $T^Y_{XZ}\neq 0$.

\vsv

Next, let us examine the behavior of the almost complex structure under the above decomposition. We have the following

\begin{lemma}
Let $(M^n,g)$ be a Hermitian manifold with $R=0$ identically, and assume that $(14)$ holds everywhere. In  an open subset of $M$ where $K_0$, $K_1'$ and $K_2'$ form distributions, we can write  $J=J_0+J_1+J_2$ for the decomposition of the almost complex structure under the decomposition $W=K_0\oplus K_1' \oplus K_2'$. Then we have
$\nabla_x J_0=\nabla_x J_1 = 0$ for any  $x\in W$, $\nabla_yJ_2=0$  for any $y\in K_2$,  and $\nabla_y J_2 \neq 0$ for any $\ 0\neq y \in K_1'$.
\end{lemma}

\begin{proof}
Under any local unitary frame $e$ in $M$, by using formula (5) and (6), we get through a straight forward computation the following:
\begin{eqnarray}
(\nabla_{e_i} J) (e_j) & = & 0 \\
(\nabla_{\overline{e_i}}J) (e_j) & = & 2\sqrt{-1} \sum_{k=1}^n T^i_{jk} \overline{e_k}
\end{eqnarray}
Then the lemma is a direct consequence of (15), (16) and the remarks above, so we will omit the details here.
\end{proof}

\vs

Now let us focus on the $3$-dimensional case. In this case we will show that equation $(14)$ always holds:

\begin{lemma}
Let $(M^3,g)$ be a compact Hermitian manifold with $R=0$ identically. Then the equality $(14)$ holds everywhere.
\end{lemma}

\begin{proof}
Since $M$ is compact and $R=0$, by the equality case of Gauduchon's inequality in \cite{Gauduchon2}, we know that $g$ is balanced. So $\sum_k T^k_{jk}=0$. By letting $k=i$ and sum up in $(12)$, we get $\sum_{r,k} T^r_{jk} T^k_{rl} =0$ for any $j$, $l$. In other words, we have
\begin{equation}
 \text{tr} (A_X A_Y) = 0 ,   \ \ \forall \ \ X, Y \in V
\end{equation}
We will show that, when $n=3$, the above equality $(17)$ actually implies $A_XA_Y=0$ for any $X$, $Y$ in $V$, which is $(14)$.

\vsv

Let $e$ be a unitary frame. Write $a_i = T^i_{jk}$, $b_i=T^j_{ij}$ where $(ijk)$ is a cyclic permutation of $(123)$. These $6$ terms are all the components of $T^c$ since $g$ is balanced. We have:
$$
A_{e_1} = \left[ \begin{array}{ccc} 0 & 0 & 0 \\ b_2 & b_1 & a_3 \\ -b_3 & -a_2 & -b_1 \end{array} \right] , \ \ A_{e_2} = \left[ \begin{array}{ccc} -b_2 & -b_1 & -a_3  \\ 0 & 0 & 0  \\ a_1 & b_3 & b_2 \end{array} \right] , \ \  A_{e_3} = \left[ \begin{array}{ccc}  b_3 & a_2 & b_1 \\ -a_1 & -b_3 &  -b_2 \\0 & 0 & 0 \end{array} \right]
$$
Therefore,
$$ \text{tr}(A_{e_i}^2)=2 (b_i^2-a_ja_k) = 0, \ \ \text{tr}(A_{e_i}A_{e_j})= 2(a_kb_k - b_ib_j) = 0 $$
where $(ijk)$ is any cyclic permutation $(123)$. Now let us fix a point $p$ and also fix $e_1$, and rotate $\{ e_2, e_3\}$ if  necessary, we may assume that $T^2_{12}=0$. That is, we may assume that $b_1=0$. The above equalities implies that $a_2a_3=b_2b_3=a_2b_2=a_3b_3=0$.

\vsv

If $a_3\neq 0$, then we have $b_3=a_2=0$. So the only possibly non-zero terms are $a_1$, $a_3$, and $b_2$. Also, $b_2^2=a_1a_3$. From this, it is easy to check that $A_{e_l}A_{e_m}=0$ for any $1\leq l, m\leq 3$. So $(14)$ holds. When both $a_2=a_3=0$, then the only possibly non-zero term would be $a_1$. In this case clearly $(14)$ holds.
\end{proof}

\vsv

So for a compact Hermitian threefold $(M^3,g)$ with $R=0$, we have the orthogonal decomposition $T_M = N_0 \oplus N_1' \oplus N_2'$ at any $p\in M$, where $T_M=V$ is the holomorphic tangent space of $M$ at $p$, and $N_i$, $N_i'$ are the complex subspace of $V$ corresponding to the real kernel spaces $K_i$, $K_i'$.

\vsv

Now let us assume that $g$ is not K\"ahler, and let $U\subseteq M^3$ be the open subset where $T^c\neq 0$. For any $p\in U$, since $N_1'$ needs to be at least one dimensional, and $N_2'$ needs to be least two dimensional, so we must have $N_0=0$ and $T_M= N_1\oplus N_2$. Let us choose a local unitary frame $e$ such that $e_3\in N_1$. Then $T^3_{12}\neq 0$ is the only non-zero components of $T^c$. By $(12)$-$(14)$, we have
$$ T^k_{ij;l}=0, \ \ T^3_{12;\overline{1}} = T^3_{12;\overline{2}} =0.$$

In the open subset $U\subseteq M$, let $V=N_1\oplus N_2$ be the decomposition of the holomorphic tangent bundle $T_M$, and $W=K_1\oplus K_2$ be the corresponding $J$-invariant orthogonal decomposition of the real tangent space of $M$. We make the following claims:

\vsv

{\bf Claim 1:} {\em In $U$, $K_2$ is a totally geodesic foliation with complete leaves.}

\vsv

{\bf Claim 2:} {\em  For any $p\in U$, the leaves of $K_2$ near $p$ are parallel to each other.}

\vsv

Fix any $p\in U$. In a small neighborhood of $U$, let $e$ be a unitary frame such that $e_3$ lies in $N_1$. This is the unique type $(1,0)$ tangent direction $X$ (up to scalar multiple) such that $T^i_{jX}=0$ for any $i$, $j$.  Denote by $\varphi$ the coframe dual to $e$. As in \cite{YZ}, write $\nabla e = \theta_1e+\overline{\theta_2}\overline{e}$ for the connection form, then the condition $R=0$ is the same as
\begin{eqnarray}
\Theta_1 & = & d\theta_1 - \theta_1  \theta_1 - \overline{\theta_2}\theta_1 \ = \ 0\\
\Theta_2 & = & d\theta_2 - \theta_2  \theta_1 - \overline{\theta_1}\theta_2 \ = \ 0
\end{eqnarray}
Since $0\neq \lambda = T^3_{12}$ is the only non-zero component of $T^c$ under $e$, by Lemma 2 of \cite{YZ}, we have
$$ \theta_2 = \left[ \begin{array}{cc} \beta E & 0 \\ 0 & 0 \end{array} \right] , \ \ \ \text{where} \ \ E = \left[ \begin{array}{cc} 0 & 1 \\ -1 & 0 \end{array} \right]  \ \ \text{and} \ \ \ \beta = \overline{\lambda } \varphi_3. $$
Let us write
$$ \theta_1 = \left[ \begin{array}{cc} \chi & \xi \\ - \xi^{\ast } & \alpha \end{array} \right]   . $$
Since $\theta_1$ is skew-Hermitian, and $E\chi +\ ^t\!\chi E=\text{tr}(\chi )E$, we get from $(18)$ and $(19)$ that
\begin{equation}
 d\chi = \chi \chi - \xi \xi^{\ast } , \ \ \ d\xi = \chi \xi + \xi \alpha ,  \ \ d\alpha = - \xi^{\ast } \xi;
 \end{equation}
\begin{equation}
 d\beta = \beta \wedge \text{tr}(\chi ), \ \ \ \beta \wedge E\xi =0.
 \end{equation}
From the second equation in $(21)$,  we know that the entries of $\xi$ are multiples of $\varphi_3$:
$$ \xi = v \varphi_3 =  \left[ \begin{array}{c} a \\ b \end{array} \right]  \varphi_3. $$
By the structure equation $d\varphi = - ^t\!\theta_1 \varphi - \ ^t\!\theta_2\overline{\varphi}$, we obtain
\begin{equation}
d\varphi_3 = - (a\varphi_1+ b  \varphi_2+\alpha )  \wedge \varphi_3 .
\end{equation}
Since $K_2$ is the distribution annihilated by $\{ \varphi_3, \overline{\varphi_3}\}$, the above identity and its conjugation show that $K_2$ is a foliation.

\vsv

To see that $K_2$ is a totally geodesic foliation, we need to show that $\langle \nabla_XY, e_3\rangle =0$ for any $X$, $Y$ in $K_2$, or equivalently,
$$ (\theta_1)_{i3}(e_j) = (\theta_1)_{i3} (\overline{e_j}) = (\theta_2)_{i3}(e_j) = (\theta_2)_{i3} (\overline{e_j}) =0$$
for any $i$, $j$ in $\{ 1, 2\}$. As  $(\theta_2)_{i3}=0$, and $(\theta_1)_{i3}$ is given by $\xi$ which is proportional to $\varphi_3$, we know that $K_2$ is a totally geodesic foliation in $U$.

\vsv

Since $ T^3_{12;k} =0$ for any $k$ and $T^3_{12;\overline{1}} = T^3_{12;\overline{2}} = 0$, we know that along any leaf of $K_2$, $\lambda$ is a constant function thus remains non-zero, so the leaves of $K_2$ are complete in $U$. This concludes the proof of Claim 1.

\vsv

Next let us prove Claim 2. It is equivalent to $K_1$ being a foliation, and equivalent to the condition that within $U$, the decomposition $W=K_1\oplus K_2$  gives a local metric product splitting. It suffices to show that $\xi =0$ at $p$.

\vsv

Let $\sigma : {\mathbb R} \rightarrow U$ be the constant-speed
geodesic contained in the leaf of $K_2$ through $p$, so that $\sigma
'(0)=e_1+\overline{e_1}$. Write $\sigma'(t)=X$. By Lemma 3,
$J_2=J|_{K_2}$ is constant along the leaves of $K_2$, so we may
choose our unitary frame $e$ in a neighborhood of $\sigma$ so that
$e_1$, $e_2$ are parallel along $\sigma$. This implies that $\chi
(X)=0$. We also have $\alpha (X)=0$ since $\overline{\alpha } = -
\alpha$, and $\varphi_1(X)=1$, $\varphi_2(X)=\varphi_3(X)=0$. The
second equation
 in $(20)$ now gives
$$ dv \varphi_3 - v (a\varphi_1 + b\varphi_2 +\alpha ) =
\chi \xi + \xi \alpha ,$$
when applied on the vectors $(X,e_3)$, we get
$$ X(a)-a^2 =0.$$
So $a(t)$ satisfies the Riccati equation along the geodesic $\sigma$.
Since solutions to the equation blows up in finite time unless the
initial condition is trivial, we know that $a$ must be zero at $p$.
Similarly, $b=0$ at $p$, and this completes the proof of Claim 2.

\vs

{\bf Claim 3:} {\em  The universal covering space $ \pi : \widetilde{M} \rightarrow M$ admits a product structure $ \widetilde{M} = Y_1\times Y_2 $, where $ Y_1\cong {\mathbb R}^2$ and $ Y_2\cong {\mathbb R}^4$, such that within the open subset $\pi^{-1}(U)$, the $Y_2$ factor are given by the leaves of $K_2$.}

\vsv

Since $M$ is a complex manifold, and it is well known that a flat metric $g$ is real analytic, any local splitting spreads to a global splitting on the universal cover. Here, however, we want to make sure that the extended splitting again respect the condition that $T^3_{12}$ is the only possibly non-zero component of $T^c$ when $e_3$ is in the $Y_1$ direction. To see this, let $\{ U_a\} _{ a\in A } $ be the connected components of $\pi^{-1}(U)$. Each $U_a$ is  isometric to the product $\Sigma_a \times L_a$ where $L_a\cong {\mathbb R}^4$, $\Sigma_a$ is an open subset of the flat ${\mathbb R}^2$ and the $L_a$ factor are given by the leaves of $K_2$.

\vsv

Given any $a$, $b\in A$, we claim that the affine subspaces $L_a$ and $L_b$ in $\widetilde{M}={\mathbb R}^6$ are parallel to each other. To this end, let $\sigma$ be a line segment in ${\mathbb R}^6$ which is the shortest path connecting $L_a$ and $L_b$. Then $\sigma$ is perpendicular to both $L_a$ and $L_b$. Consider the tangent vector field $X=\sigma'(t)$ along $\sigma$. Within $U_a$, as $X$ lives in $K_1$, $JX$ is parallel along $\sigma \cap U_a$ by Lemma 3. Since $g$ is  real analytic, $JX$ is parallel along the entire $\sigma$. Now as both $L_a$ and $L_b$ are perpendicular to $X$ and $JX$, they must be parallel to each other.

\vsv

Note that by Claim 3 and Lemma 3, we know that the complex structure on $\widetilde{M}$ is actually a warped complex structure in the sense of \cite{BSV}, namely, if we write $J=J_1+J_2$ for the decomposition of the almost complex structure, then $J_1$ is constant, and makes $Y_1$ the flat ${\mathbb C}$, and at any $(y_1, y_2) \in \widetilde{M}$, $J_2$ is given by $J_{f(y_1)} \in Z_2$ where $Z_2=SO(4)/U(2)\cong {\mathbb P}^1$ is the space of all complex structures on ${\mathbb R}^4$ compatible with the metric and the orientation, and $f: Y_1\cong {\mathbb C} \rightarrow Z_2$ is a smooth map. As proved in \cite{BSV}, the integrability of $J$ corresponds to the holomorphicity of $f$. (See also the next section for an explicit calculation of this). Clearly, when the flat metric $g$ is not K\"ahler with respect to $J$, $f$ can not be a constant.

\vs

{\bf Claim 4:} {\em The leaves of $K_2$ are compact in $M$.}

\vsv

Let us denote by $\Gamma$ the deck transformation group of $M$. Replacing $M$ by a finite unbranched cover of it if necessary, we may assume that $\Gamma \cong {\mathbb Z}^6$ acting as translations in ${\mathbb R}^6$.  For $i=1$, $2$, let $p_i: \Gamma \rightarrow \Gamma_i$ be  the projection into the isometry group of the factors $Y_i$, with $\Gamma_i$ being the image.

\vsv

For any $\gamma (y_1, y_2) = (y_1+a, y_2+b)$ in $\Gamma$, since the
complex structure on $\widetilde{M}$ is preserved by $\gamma$, we
have $J_{f(y_1)} = J_{f(y_1+a)}$, where $f: Y_1={\mathbb C}
\rightarrow Z_2={\mathbb P}^1$ is the holomorphic map characterizing
$J$ as a warped complex structure. That says that any
$\Gamma_1$-orbit is contained in a level set of $f$, which is
necessarily discrete in $Y_1={\mathbb C}$. So $\Gamma_1$ is
discrete, which will imply that the leaves of $Y_2$ close up in $M$.

Indeed, let us take a leaf $F$ of the foliation of $Y_2$ in $M$, if
$F$ is not compact, then there will be a sequence $x_i$ in $F$ that
converges to a point $x_0 \in M$, such that $x_0 \not\in F$. Take a
sufficiently small neighborhood $U$ of $x_0$, inside $U$ the
foliation can be parameterized by $F_t$, where $t$ belongs to a
small open subset $V \subset Y_1$. We may assume that $F_0$ is the
one through $x_0$. By assumption $F_0$ is not in $F$, but there
exists $t_i \rightarrow 0$ such that $F_{t_i}$ is a part of $F$.

Now let us look at the picture on the universal cover. Take a point
$0$ over $x_0$ and a small neighborhood $\tilde{U}$ over $U$. The
pre-image $\pi^{-1}(F)$  is equal to the union of $\Gamma_1 \times
Y_2$. So if $\Gamma_1$ is discrete, then $\pi^{-1}(F)$ would be
closed in the universal cover, However in $\tilde{U}$, we have the
same picture of $F_{t_i}$ and $F$ as in $U$. This leads to a
contradiction.

\vsv

To summarize, we have proved that, if  $(M^3,g)$ is a compact,
non-K\"ahler, Hermitian manifold with flat Riemannian connection,
then a finite unbranched cover $M'$ of $M$ is isometric to $M_1\times M_2$,
where $(M_1,g_1)$ is a flat $2$-torus and $(M_2,g_2)$ is a flat $4$-torus,
and the complex structure $J$ on $M'$ is given by
$$ J=J_1+ J_{f(x_1)}$$
at the point $(x_1,x_2)\in M'$, where $J_1$ is a constant complex structure
 on $M_1$ compatible with $g_1$ and makes $M_1$ an elliptic curve, and
 $f: M_1 \rightarrow Z_2\cong {\mathbb P}^1$ is a holomorphic map from
 the elliptic curve into the space of oriented orthogonal complex structures
  on ${\mathbb R}^4$. In other words, $M'$ is a BSV $3$-torus.  This
  completes the proof of Theorem 1.

\vs

\vs

\section{The BSV-tori in dimension three}\label{BSV systematic}

\vsv

In this section, let us give a more detailed discussion on the BSV-tori in dimension three, and show that they are indeed non-K\"ahlerian, namely, as a complex manifold they do not admit any K\"ahler metric.  The readers are referred to \cite{BSV} for a much broader discussion on the subject, and here we will try to be explicit and also focus on the differential-geometric aspect.

\vsv

Following \cite{BSV}, let $Z_2$ be the set of all constant complex structures on ${\mathbb R}^4$ compatible with a fixed flat metric and orientation. Its elements are skew-symmetric orthogonal $4\times4$ real matrices, and with a choice of orientation, they can be expressed as
\begin{equation}
 J_{(a,b,c)}=\left[ \begin{array} {ll} aE & bE+cI \\ bE-cI & -aE \end{array}
 \right] , \ \ \ \ \mbox{where} \ \
 E= \left[ \begin{array} {cc} 0 & 1 \\ -1 & 0 \end{array} \right] ,
 \label{J representation}
\end{equation}
$I$ is the identity matrix, and $a^2+b^2+c^2=1$. Under the identification $S^2\cong {\mathbb P}^1={\mathbb C} \cup \{ \infty \}$, we have
\begin{equation}
 a=\frac{2x}{r^2+1}, \ \ b = \frac{2y}{r^2+1}, \ \ c=\frac{r^2-1}{r^2+1}, \ \ \ \mbox{where} \ \ r=|z|, \ z=x+iy \in {\mathbb C}\cup \{ \infty \}.
 \end{equation}
We will write the above $J_{(a,b,c)}$ simply as $J_z$.

\vsv

Now suppose that $(M_1,J_1,g_1)$ is a compact Hermitian manifold, and   $f: M\rightarrow {\mathbb P}^1$ a smooth map. Let $(M_2,g_2)$ be a flat $4$-torus, and consider the manifold $M=M_1\times M_2$, equipped with the Riemannian product metric $g=g_1\times g_2$, and the warped almost complex structure $J$ on $M$ giving by
\begin{equation}
J = J_1 + J_{f(y_1)}
\end{equation}
at $(y_1,y_2)\in M$.  Clearly, $J$ is orthogonal with respect to $g$, and as proved in \cite{BSV} and also in \cite{Burstall1}, the integrability of $J$ is equivalent to the holomorphicity of the map $f$. Let us verify the equivalence in this explicit special case, namely, let us prove the following

\begin{lemma}
The almost complex structure $J$ defined on $M=M_1\times M_2$ as above is integrable if and only if the map  $f: M_1\rightarrow {\mathbb P}^1$ is holomorphic.
\end{lemma}

\begin{proof}
As is well known, $J$ is integrable if and only if the Nijenhuis tensor
\begin{equation}
N_J(X,Y) : = [X,Y] -  [JX,JY] + J[JX,Y] + J[X,JY]
\end{equation}
vanishes identically, for any vector fields $X$, $Y$ in $M$. Since $[X,Y]=\nabla_XY-\nabla_YX$, $J_1$ is integrable, and $J$ is constant along the $M_2=T^4_{\mathbb R}$ direction, we get $N_J(X,Y)=0$ if $X$, $Y$ are both in the $M_1$ direction or both in the $M_2$ direction. So it suffices to verify $N_J(X,Y)=0$ for $X$ in $M_1$ and $Y$ in $M_2$. Since $\nabla_YJ_1=0$ and $\nabla_{JY}J_1=0$, we get
$$ N_J(X,Y) = \nabla_XY + J\nabla_{JX}Y -\nabla_{JX}JY + J\nabla_XJY.$$
Now let $\{ \epsilon_1, \ldots , \epsilon_4\}$ be the standard parallel frame on $M_2$. By taking $Y$ to be any $\epsilon_i$ in the above equality, we know that $N_J$ vanishes on $M$ if and only if
\begin{equation}
\nabla_{J_1\!X}J\epsilon_i = J \nabla_X J\epsilon_i
\end{equation}
for any $1\leq i\leq 4$ and any tangent vector $X$ in $M_1$. Let us write $J=J_z$ the $4\times 4$ matrix in $(23)$, where $z=f(y_1)$, and denote by $J'$, $\dot{J}$ its derivative
in the direction $J_1\!X$, $X$, respectively.  Then the identity $(27)$ is simply
\begin{equation}
J'=J\dot{J}.
\end{equation}
Using the expression of $J$ in $(23)$, and the fact $a\dot{a}+b\dot{b}+c\dot{c}=0$, we get
\begin{equation}
\left\{ \begin{array}{lll} a' = c\dot{b} -b\dot{c} \\ b' = a\dot{c} -c\dot{a} \\ c' = b\dot{a} -a\dot{b} \end{array} \right.
\end{equation}
Now if we  use the coordinate $z=x+iy$ and the stereographic projection formula $(24)$, then it is a straight forward computation to see that the above system is equivalent to the following
\begin{equation}
\left\{   \begin{array} {ll}  x'  =  \dot{y}  \\  y'  = -  \dot{x}  \end{array}  \right. ,
\end{equation}
which is just the Cauchy-Riemann equation. So $J$ is integrable if and only if the map $z=f(y_1)$ is holomorphic. \end{proof}

Of course the torus $T^4_{\mathbb R}$ in above lemma can be replaced by $T^{2k}_{\mathbb R}$ for any $k\geq 2$, and the lemma is still valid. This is Proposition 5.2 in \cite{BSV} or  Proposition 5.1 in \cite{Burstall1}.

\vsv

Following \cite{BSV}, we will call the above compact Hermitian manifold $(M,g,J)$ a {\em warped torus}, and we are particularly interested in the complex dimension three case, namely, when $M_1$ is a compact Riemann surface of genus $g(M_1)$, and $f$ is a non-constant holomorphic map from $M_1$ into ${\mathbb P}^1$, or equivalently, a non-constant meromorphic function on the curve $M_1$.   We will denote this compact Hermitian threefold by $M^3_f$.

\vsv

Note that for $g(M_1)=1$ and $f$ non-constant, these $M^3_f$ are the BSV $3$-tori defined in \cite{BSV}. When $M_1={\mathbb P}^1$ and $f$ is the identity map $\iota $, then $M_{\iota}$ is the twistor space over the flat $4$-torus $M_2$. For $g(M_1)\geq 2$, such $M^3_f$ include Calabi's pioneer construction in \cite{Calabi}.

\vsv

As a complex manifold, it is clear that the projection map $\pi_1: M^3_f \rightarrow M_1$ is a holomorphic submersion, and the fibers are flat complex $2$-tori, but are not isomorphic to each other in general, so $\pi_1$ is not a holomorphic fiber bundle. For any $y_2\in M_2$, the subset $C_{y_2}=M_1\times \{ y_2\} $ is a totally geodesic complex submanifold of $M^3_f$  and is holomorphically isometric to $M_1$, but $C_{y_2}$ does not vary holomorphically in $y_2\in M_2$.

\vsv

It seems that these $M^3_f$ form a rather interesting class of complex threefolds, and here we will satisfy ourselves by exploring their Hermitian geometry a little bit, and showing that they are always non-K\"ahlerian (for non-constant $f$).

\vsv

First let us choose a convenient local unitary frame $e$ on $M^3_f$. Let $f$ be any non-constant meromorphic function on $M_1$, and write $V_0=M_1\setminus \{ f=\infty \}$, $V_{\infty }=M_1\setminus \{ f=0 \}$. Let $D_1, \ldots , D_m$ be open subsets in $M_1$ such that their union is the entire $M_1$, and on each $D_j$ there exists a $(1,0)$-form $\psi_j$ with unit norm. Then the open subsets
 $$ U_{j0}=\pi_1^{-1}(D_j\cap V_0), \ \ U_{j\infty } = \pi_1^{-1}(D_j\cap V_{\infty }), \ \ 1\leq j\leq m, $$
 form an open covering of $M^3_f$. On each $U_{j0}$, we have a unitary coframe $\varphi$ where $\varphi_3=\pi_1^{\ast }\psi_j$, and
 \begin{eqnarray}
 \varphi_1 & =  & \frac{1}{\sqrt{2}\sqrt{1+|f|^2}} \{ f (dx_1- idx_3) + i (dx_2-idx_4)\} \\
 \varphi_2 & =  & \frac{1}{\sqrt{2}\sqrt{1+|f|^2}} \{ -i  (dx_1 + idx_3) + f (dx_2 + idx_4)\}
 \end{eqnarray}
at the point $(y_1, y_2)$ in $U_{j0}$, where $f=f(y_1)$, and $( x_1, \ldots , x_4)$ is the standard Euclidean coordinate on the universal cover $\widetilde{M_2}={\mathbb R}^4$. Note that away from the poles of $f$, the above expressions are well-defined, and it is easy to check that $\varphi$ is indeed unitary and of type $(1,0)$ as $J$ is defined by $(23)$-$(25)$. In each $D_{j\infty }$, a coframe can be given in a similar fashion, which we will omit.

\vsv

In $D_j$, we have the structure equation $d\psi_j=-\xi \psi_j$, $d\xi =\Xi$, where $\Xi$ is the curvature form of $M_1$. Under the unitary coframe $\varphi$ in $U_{j0}$, it is easy to see that the connection forms are given by
\begin{equation}
\theta_1 = \left[ \begin{array}{ccc}
\alpha & 0 & 0 \\ 0 & \alpha & 0 \\ 0 & 0 & \pi_1^{\ast } \xi  \end{array}
\right] , \ \ \ \theta_2 =
\left[ \begin{array}{ccc} 0 & \beta & 0 \\ -\beta & 0 & 0 \\ 0 & 0 & 0 \end{array}
\right] , \ \ \
\theta = \left[ \begin{array}{ccc}
\alpha & 0 & -\lambda \varphi_2 \\ 0 & \alpha & \lambda \varphi_1 \\
\overline{\lambda} \overline{\varphi_2} & -\overline{\lambda}
\overline{\varphi_1} & \pi_1^{\ast } \xi  \end{array} \right],
\label{two connection 1-forms}
\end{equation}
where $\lambda = T^3_{12}$, and
$$ \alpha = \frac{1}{2(1+|f|^2)} (fd\overline{f} - \overline{f} df), \ \ \ \ \beta = \overline{\lambda }\varphi_3 = - \frac{i}{1+|f|^2} df. $$
We have $d\alpha = \beta \overline{\beta}$, $d\beta = 2\beta \alpha$. From the structure equation $d\varphi = - \ ^t\! \theta_1 \varphi - \ ^t\! \theta_2 \overline{\varphi}$, we get
$$ d\varphi_1 = -\alpha \varphi_1+\beta \overline{\varphi_2}, \ \ \  d\varphi_2 = -\alpha \varphi_2-\beta \overline{\varphi_1}, \ \ \  d\varphi_3 = - \pi_1^{\ast }\xi \ \varphi_3 .$$
By taking exterior differentiation of $\beta = \overline{\lambda}\varphi_3$, we get
$$ (d\overline{\lambda} + 2\overline{\lambda} \alpha - \overline{\lambda} \ \pi_1^{\ast }\xi ) \wedge \varphi_3 = 0 , $$
so there will be a local smooth function $\mu$ in $U_{j0}$ such that
\begin{equation}
d\lambda - 2\lambda \alpha + \lambda \pi_1^{\ast } \xi =  \mu \overline{\varphi_3}.
\end{equation}
We compute the curvature form of the Chern connection $\Theta =
d\theta - \theta \wedge \theta$ as follows:
\begin{equation}
\Theta =    \left[ \begin{array}{ccc} |\lambda|^2 (\varphi_2\overline{\varphi_2} + \varphi_3\overline{\varphi_3})  & - |\lambda|^2 \varphi_2\overline{\varphi_1} & |\lambda|^2 \varphi_3\overline{\varphi_1} -\mu \ \varphi_2\overline{\varphi_3}  \\
- |\lambda|^2 \varphi_1\overline{\varphi_2} & |\lambda|^2 (\varphi_1\overline{\varphi_1} + \varphi_3\overline{\varphi_3})  &  |\lambda|^2 \varphi_3\overline{\varphi_2} + \mu \ \varphi_1\overline{\varphi_3} \\
|\lambda|^2 \varphi_1\overline{\varphi_3} - \overline{\mu } \ \varphi_3\overline{\varphi_2}   &  |\lambda|^2 \varphi_2\overline{\varphi_3} + \overline{\mu } \ \varphi_3\overline{\varphi_1}  & \pi_1^{\ast } \Xi   - |\lambda|^2 (\varphi_1\overline{\varphi_1} + \varphi_2\overline{\varphi_2}) \end{array} \right]  .
\end{equation}
From this, one gets the Chern forms of $M$, and thus the Chern classes. It is easy to see that
\begin{equation}
c_1(M)=2(1+\text{deg}(f) -g(M_1)) \ \pi_1^{\ast} \sigma , \label{c1
formula}
\end{equation}
where  $\sigma \in H^2(M_1, {\mathbb Z})\cong {\mathbb Z}$ is the
positive generator. In particular, $c_1(M)=0$ if and only if
$g(M_1)=1+\text{deg}(f)$, and such example with the lowest genus
would be $g(M_1)=3$ and $\text{deg}(f)=2$, namely, a hyperelliptic
curve of genus $3$. This includes Calabi's 3-folds constructed in
\cite{Calabi}. In Section \ref{Calabi 3-folds revisited} we will
give a detailed discussion on the connection between Calabi's
3-folds and BSV type warped complex structures.

\vsv

Since $|T^c|^2=8\sum_{i,j,k} |T^k_{ij}|^2$, in the case of $M^3_f$, it is equal to $16|\lambda|^2$, and we have
$$ |\lambda |^2\varphi_3\overline{\varphi_3} = \beta \overline{\beta } = \frac{dfd\overline{f}}{(1+|f|^2)^2}. $$
On ${\mathbb P}^1={\mathbb C}\cup \{ \infty \}$, we have
$$ \int_{{\mathbb P}^1} \frac{idzd\overline{z}} {(1+|z|^2)^2} = 2\pi .$$
So the  $L^2$-norm of the Chern torsion of $M^3_f$, or its {\em total Chern torsion,}  is given by
\begin{equation}
\int_M |T^c|^2dv = 32\pi v_2 \text{deg}(f), \label{total torsion
norm}
\end{equation}
where $v_2$ is the volume of $M_2$. In other words, the total Chern torsion of $M^3_f$ can be arbitrarily large, when $\text{deg}(f)$ gets bigger and bigger.

\vsv

Next, let us show that $M^3_f$ is not K\"ahelrian, namely, it cannot admit any K\"ahler metric. To see this, let us compute
$$ d(\varphi_1\overline{\varphi_1})=(-\alpha \varphi_1 + \beta \overline{\varphi_2})\overline{\varphi_1} - \varphi_1 (-\overline{\alpha }\overline{\varphi_1} + \overline{\beta } \varphi_2) = -\beta \overline{\varphi_1} \overline{\varphi_2} + \overline{\beta } \varphi_1 \varphi_2,$$
here we used the fact that $\overline{\alpha } = - \alpha$, thus we get
$$
 \partial \overline{\partial } (\varphi_1\overline{\varphi_1}) =
 d \overline{\partial } (\varphi_1\overline{\varphi_1}) =
 d (-\beta \overline{\varphi_1} \overline{\varphi_2} ) =
 \beta \overline{\beta} (\varphi_1 \overline{\varphi_1} +
 \varphi_2 \overline{\varphi_2}).
 $$
Similarly,
$$  \partial \overline{\partial } (\varphi_2\overline{\varphi_2})
= \beta \overline{\beta} (\varphi_1 \overline{\varphi_1} +
\varphi_2 \overline{\varphi_2}) ,$$
Therefore, we get the following
\begin{equation}
\partial \overline{\partial } \omega_g = 2 \sqrt{-1} \beta
\overline{\beta} (\varphi_1 \overline{\varphi_1} + \varphi_2
\overline{\varphi_2}) = 2 \beta \overline{\beta}  \ \omega_g.
\end{equation}
Now, if $\omega_h$ is a Hermitian metric on $M^3_f$. Write $\omega_h
= \sqrt{-1} \sum h_{i\overline{j}} \varphi_i \overline{\varphi_j}$.
The matrix $(h_{i\overline{j}})$ is positive definite. We have
\[
\sqrt{-1} \partial \overline{\partial } \omega_g \wedge \omega_h =
\frac{1}{3} |\lambda|^2 (h_{1\overline{1}}+ h_{2\overline{2}}) \
\omega_g^3. \]
Clearly, the integral of the right hand side over
$M^3_f$ is positive, therefore we conclude that $\omega_h$ cannot
satisfy the condition $\partial \overline{\partial } \omega_h =0$
everywhere. That is, we have

\begin{lemma}
Let $(M_1, J, g_1)$ be any compact Riemann surface and $f$
any non-constant holomorphic map from $M_1$ to $\mathbb{P}^1$. Then the
warped complex tori $M^3_f=M_1 \times T^4_{\mathbb{R}}$ does not
admit any Hermitian metric that is pluri-closed. In particular, any
such $M^3_f$ is non-K\"ahlerian.
\end{lemma}

The notion of \emph{G-K\"ahler-like} was introduced in \cite{YZ} and
it is equivalent to $\Theta_{2}=0$ (Lemma 5 in \cite{YZ}). One
result in \cite{YZ} implies any compact G-K\"ahler-like Hermitian
manifold must be balanced. Now it follows from (\ref{two connection
1-forms}) that the product metric $g_1 \times g_2$ on $M^3_f$ has
$\Theta_{2}=0$, hence is G-K\"ahler-like and balanced. We remark
that this observation is also implied by a more general result in
\cite{BSV}. (See Proposition 5.3(ii) on P.144 in \cite{BSV})

\vsv

Next let us show that $\text{kod}(M^3_f)=-\infty$ when the base $M_1$
is an elliptic curve, i.e. for any $m \geq 1$, any $s \in
H^{0}(M^3_f, m K_M)$ must be identically $0$.

\vsv

If not, let $D$ be the zero locus of $s$, then $D$ is an effective
divisor in $M$. Since $(M^3_f,g_1 \times g_2)$ is balanced, the
integral of $\omega^2$ along $D$ is well-defined, which will be the
volume of $D$, thus positive. If $s$ is nowhere zero, then $D$ is
the zero divisor and this integral is zero.

On the other hand, we have:
\[
\int_{D} \omega^2= \int_{M^3_f} -m
\frac{\sqrt{-1}}{2\pi}\operatorname{Tr}(\Theta) \wedge
\omega^2=\int_{M^3_f} \frac{-2m}{\pi} |\lambda|^2
\frac{\sqrt{-1}}{2} \varphi_3 \overline{\varphi_3} \wedge
\omega^2=\frac{-m}{3\pi} \int_{M^3_f} |\lambda|^2 \omega^3,
\] so the integral is always negative, a contradiction.

\vsv

Note that for compact Riemannian surfaces $M_1$ of genus $g_1$, the last integral in the above equals to $-2m (1-d-g_1)v_2$, where $d$ is the degree of $f: M_1 \rightarrow {\mathbb P}^1$ and $v_2$ the volume of the $4$-torus. So the Kodaira dimension of $M_f^3$ will be $-\infty $ if $1+d-g_1>0$. Note that $d$ is always a positive integer, and it can be $1$ only when $g_1=0$, so for $g_1\leq 2$ one always has $1+d-g_1>0$.

\vsv

The above discussion is summarized in the following lemma.
\begin{lemma}
The warped complex tori $M^3_f=M_1 \times
T^{4}_{\mathbb{R}}$ have Kodaira dimension $-\infty$ when the genus of $M_1$ is $2$ or less. In particular, this is the case for all $BSV$ $3$-tori.
\end{lemma}

This completes the proof of Proposition \ref{torsionparallel}.

\vs

\section{Calabi 3-folds revisited} \label{Calabi 3-folds revisited}

In 1958 Calabi \cite{Calabi} discovered that $M_1 \times
T_{\mathbb{R}}^4$ where $M_1$ is a hyperelliptic Riemann surface
with odd genus $g \geq 3$ and $T^4$ a real $4$-torus, can be given a
complex structure $J$ such that the resulting threefold $(M_3, J)$
admits no K\"{a}hler metric and has vanishing fist Chern class.
In this section, we explore the connection between Calabi's
construction and the BSV type warped complex structures.

\vsv

\subsection{A review of Calabi's 3-folds} \label{review Calabi}

Without specification, all results in this subsection is from Calabi
\cite{Calabi}. Let $M_1$ be a hyperelliptic Riemann surface with
odd genus $g \geq 3$, then $M_1$ admits a meromorphic function of
degree 2, branched over $2g+2$ distinct points on $M_1$. Denote these points by
$P_i$, and assume that $z(P_i) \neq \infty$ for each $i$. Then we get a
single-valued meromorphic function on $M_1$:
$$w=\sqrt{\prod_{i=1}^{2g+2} (z-z(P_i))}.$$

It is well-known that $H^{1,0}(M_1)=\operatorname{Span}
\{\frac{z^j}{w}dz \mid 0\leq j \leq g-1\}$. Let $\phi(z)$
be an arbitrary polynomial in $z$ of degree $\frac{g-1}{2}$. If we
view $\phi$ as a meromorphic function on $M_1$, it is of degree
$g-1$.

\vsv

Now pick the following three linearly independent forms from
$H^{1,0}(M_1)$
\[
\omega_1=\frac{\phi^2(z)-1}{2w} dz, \omega_2=\frac{\phi(z)}{w} dz,
\omega_1=\frac{\phi^2(z)+1}{2\sqrt{-1}w} dz.
\]
This is exactly the Weierstrass representation of minimal surfaces,
since $\omega_1, \omega_2, \omega_3$ do not vanish simultaneously on
$M_1$. This implies that the map $(x_1, x_2, x_3)$ where
\[
x_i= \operatorname{Re} \int_{Q_0}^{Q} \omega_i,
\]
locally maps $M_1$ to a minimal surface in $\mathbb{R}^3$. Here
$Q_0$ is a fixed point on $M_1$. In general, those $x_i$ are not
well-defined on $M_1$ globally, and they depend on $\pi_{1}(M_1, Q_0)$. But after lifting the map to the maximal Abelian covering $\widetilde{M_1}$ of $M_1$, one
gets a minimal immersion $F_1: \widetilde{M_1} \rightarrow
\mathbb{R}^3$. The image $F_1(\widetilde{M_1})$ might be complicated
(e.g., everywhere dense), but the covering transformations in
$\widetilde{M_1}$ are ${\mathbb Z}^{g}$ generated by translations in
$\mathbb{R}^3$.

\vsv

Define $F \doteq F_1 \times Id: \widetilde{M_1} \times \mathbb{R}^4
\rightarrow \mathbb{R}^7$, then $F$ defines an immersed hypersurface
in $\mathbb{R}^7$. Note that the space of purely Cayley numbers can
be identified as $\mathbb{R}^7$, therefore any immersed hypersurface
in $\mathbb{R}^7$ can be made an almost complex manifold by defining
\[
dF(J u)=N \times dF(u)
\] for any $u \in T (\widetilde{M_1} \times
\mathbb{R}^4)$. Here $\times$ stands for the cross product defined on the space of
purely Cayley numbers.

\vsv

Calabi \cite{Calabi} proved that the almost complex structure on the image of $F_1 \times Id$ in
$\mathbb{R}^7$ is integrable when $F_1$ is a minimal immersion.
Moreover, such a complex structure is invariant under translations
in $\mathbb{R}^7$, thus descends down to the compact quotient $M_1 \times
T^{4}_{\mathbb{R}}$. This is the Calabi's 3-fold $M^3$. It is proved
in \cite{Calabi} that $M^3$ admits no K\"ahler metric and has
$c_1(M^3)=0$.

\vsv

Calabi also proved that the induced metric $ds^2=dF \cdot dF$
from $F(\widetilde{M_1} \times \mathbb{R}^4) \in \mathbb{R}^7$ is
compatible with the complex structure $J$ defined above. Therefore the corresponding metric
$g$ on $(M^3, J)$ is a Hermitian metric.

\vsv

Gray \cite {Gray} studied the curvature properties of Calabi's 3-fold
$(M^3, J, g)$. It is shown in \cite {Gray} that Calabi's metric is
G-K\"ahler-like (see \cite{YZ} for a definition.)

\vsv

\subsection{Calabi 3-folds in terms of BSV-warped complex structures}

Let $(M^3, J)$ be the Calabi's 3-fold in Subsection \ref{review
Calabi}, we now explain that Calabi's construction is exactly a
special case of BSV-warped complex structure on $M_1 \times
T^4_{\mathbb{R}}$.

\vsv

By a direct calculation, one sees that the image $F_1(\widetilde{M_1})$ in
$\mathbb{R}^3$ defined above has the unit normal vector
\[
N=(N_1, N_2, N_3)=(\frac{2 \operatorname{Re}\phi}{|\phi|^2+1},
\frac{ |\phi|^2-1}{|\phi|^2+1}, \frac{2
\operatorname{Im}\phi}{|\phi|^2+1}).
\]
Therefore, the corresponding image $F(\widetilde{M_1} \times
\mathbb{R}^4)$ in $\mathbb{R}^7$ has the unit normal vector
$$ N = ( N_1, N_2, N_3, 0, 0, 0, 0).$$

Using the table of the cross product defined on purely Cayley
numbers, it is straightforward to write down the action of $J$
restricted on $T_p (\widetilde{M_1})$ and on $T_p
({\mathbb R}^4)=\operatorname{Span}\{\frac{\partial}{\partial x_4},
\frac{\partial}{\partial x_5}, \frac{\partial}{\partial x_6},
\frac{\partial}{\partial x_7}\}$. For example, the first one is
independent of $T_p ({\mathbb R}^4)$, while the latter takes the following
form under the basis $\{\frac{\partial}{\partial x_4},
\frac{\partial}{\partial x_5}, \frac{\partial}{\partial x_6},
\frac{\partial}{\partial x_7}\}$:
\[
\left[ \begin{array} {cccc}  0 & N_1 & N_2 & N_3 \\
-N_1 & 0 & -N_3 & N_2\\
-N_2 & N_3 & 0 & -N_1\\
-N_3 & -N_2 & N_1 & 0
\end{array} \right]
\]
Comparing with Formula (\ref{J representation}), we see that
Calabi's complex structure can be written as $J=J_1+J(\phi(x))$
where $J_1$ is the complex structure on $M_1$ and $J(\phi(x))$ is
the complex structure on $T^{4}_{{\mathbb R}}$ defined by the
holomorphic map $\phi: M_1 \rightarrow {\mathbb P}^1$. Note that
$\phi$ is of degree $g-1$, by (\ref{c1 formula}) we also see that
$c_1 (M^3)=0$.

\vsv

Another interesting formula from the Weierstrass representation is
that the total curvature $\int_{F_1(M_1)} |K|^2 dA$ equals to $4\pi$
multiple the degree of the Gauss map determined by $N$ (i.e. degree
of $\phi$). This resembles (\ref{total torsion norm}) on the total
Chern torsion.

\vsv

\subsection{Remarks on the induced Hermitian metric on the Calabi 3-folds}

As mentioned in Subsection \ref{review Calabi}, Calabi defined the
Hermitian metric $g$ on $M_1 \times T^4$ as the one on
$F(\widetilde{M_1} \times {\mathbb R}^4)$ induced from the standard Euclidean
metric on ${\mathbb R}^7$. Note that
\[
F(\tilde{Q}, x_4, x_5, x_6, x_7)=(x_1, x_2, x_3, x_4, x_5, x_6, x_7)
\] where $\tilde{Q} \in \widetilde{M_1}$. Therefore, the induced
metric $g$ is of the form (here $u$ is a local
holomorphic coordinate on $\widetilde{M_1}$)
\[
g=\sum_{i=1}^{3}|\frac{\partial x_i}{\partial u}|^2+\sum_{i=4}^7
|dx_i|^2.
\]
Apply the formula for the Weierstrass representation, we have
\[
g=\frac{(|\phi|^2+1)^2}{8|w|^2} |dz|^2+\sum_{i=4}^7 |dx_i|^2
\]
This is the metric that Gray \cite{Gray} proved to be
G-K\"ahler-like. It is a product Riemannian metric, however, in
general the factor on the $M_1$ direction is not the standard
hyperbolic metric. Indeed its Gauss curvature has the formula
\[
K=-2[\frac{4|\phi^{\prime} w|}{(|\phi|^2+1)^2}]^2.
\]
It has constant Gauss curvature if and only if there exists a
constant $c>0$ such that $c|\phi^{\prime}w|=(|\phi|^2+1)^2$.
However, there exists no such $\phi$ which is a polynomial of degree
$\frac{g-1}{2}$ in terms of $z$. This can be seen by comparing the
growth near $z=\infty$ determined by the degree.

\vsv

Motivated by the above discussion, we would like to raise the following
question:

\begin{question}
Let $(M_1,g_1)$ be a compact Riemannian surface of genus $g(M_1)\geq
2$ equipped with the hyperbolic metric, and let
$(T^4_{\mathbb{R}},g_2)$ be a flat $4$-torus. What is the space of
all orthogonal complex structures on $M_1 \times T^4_{\mathbb R}$
with respect to $g_1 \times g_2$?
\end{question}

\vsv

In particular, one would like to know the subset with vanishing first Chern class.

\vs

\noindent\textbf{Acknowledgments.} We would like to thank Sagun
Chanillo, Bo Guan, Xiaojun Huang, Kefeng Liu, Hongwei Xu, and Xiaokui Yang
for their interests and encouragement. We would also like to thank
Qingsong Wang for many helpful comments on an earlier version of the
manuscript.

\vs


\begin{thebibliography}{99}                     %


\bibitem {Boothby} W. Boothby, \emph{Hermitian manifolds with zero
curvature.} Michigan Math. J. {\bf 5}, (1958), no. 2, 229--233.

\bibitem {BSV} L. Borisov, S. Salamon, and J. Viaclovsky, \emph{Twistor geometry and warped product orthogonal complex structures,}
Duke Math. J. {\bf 156},  (2011), 125--166.


\bibitem {Burstall} F. Burstall, O. Muskarov, G. Grantcharov and J. Rawnsley, \emph{Hermitian structures on Hermitian symmetric spaces,}
 J. Geom. Phys. {\bf 10},  (1993),  245--249


\bibitem {Burstall1} F. Burstall, S. Gutt and J. Rawnsley, \emph{Twistor spaces for Riemannian symmetric spaces,}
 Math. Ann. {\bf 295}, (1993), 729--743.




\bibitem {Calabi} E. Calabi,  \emph{Construction and properties of some 6-dimensional almost complex manifolds,}
Trans. Amer. Math. Soc. {\bf 87}, (1958), 407--438.


\bibitem {Calabi-Eckmann} E. Calabi and A. Eckmann,  \emph{A class of compact, complex manifolds which are not algebraic,} Ann. Math. (2)
 {\bf 58}, (1953), 494--500.

\bibitem {Fu} J.X. Fu,  {\em On non-K\"ahler Calabi-Yau threefolds with balanced metrics.} Proceedings of the International
Congress of Mathematicians. Volume II, 705-716, Hindustan Book Agency, New Delhi, 2010.




\bibitem {Gauduchon} P. Gauduchon, \emph{La {$1$}-forme de torsion d'une
vari\'et\'e hermitienne compacte.} Math. Ann. 267 (1984), no. 4,
495-518.

\bibitem {Gauduchon1} P. Gauduchon, \emph{ Complex structures on compact conformal manifolds of negative type,} In:
Complex Analysis and Geometry, Lecture Notes Pure and Appl. Math. 173 Dekker, New York, 1996, pp 201-212.



\bibitem {Gauduchon2} P. Gauduchon, \emph{Hermitian connections and Dirac operators,} Boll. Un. Mat. Ital. B (7) {\bf 11} (1997), no. 2, suppl., 257-288.

\bibitem {Gray} A. Gray, \emph {Curvature identities for Hermitian and
almost Hermitian manifolds.} Tohoku Math. J. (2) {\bf 28} (1976),
no. 4, 601--612.

\bibitem {Hernandez} L. Hernandez,  \emph {Curvature vs. almost hermitian structures,} Geometriae Dedicata {\bf 79}, (2000), 205-218.

\bibitem {Salamon} S. Salamon. \emph{Harmonic and holomorphic maps, } in: Proc. Geometry Seminar Luigi Bianchi II, ed. E. Vesentini (Springer, Heidelberg, 1985).


\bibitem {Salamon1} S. Salamon, \emph{Orthogonal complex structures,} Differential geometry and applications (Brno, 1995), 103-117, Masaryk Univ., Brno, 1996.

\bibitem {SV} S. Salamon and J. Viaclovsky, \emph{Orthogonal complex structures on domains in ${\mathbb R}^4$, }  Math. Ann. {\bf 343} (2009), no. 4, 853-899.

\bibitem {Wang} H.-C. Wang, \emph{Complex parallisable manifolds,} Proc. Amer. Math. Soc. {\bf 5} (1954), 771-776.



\bibitem {WYZ} Q. Wang, B. Yang and F. Zheng, \emph{On Bismut flat manifolds,}
preprint, arXiv:1603.07058.

\bibitem {YZ} B. Yang and F. Zheng, {\em On curvature tensors of
Hermitian manifolds,} preprint, arXiv:1602.01189.


\bibitem {Zheng} F. Zheng, {\em Complex differential geometry.} AMS/IP Studies in Advanced Mathematics, 18. American
Mathematical Society, Providence, RI; International Press, Boston,
MA, 2000.








\end{thebibliography}
\end{document}